\documentclass[a4paper,12pt]{amsart}
\usepackage{amsmath, amsthm, amssymb, amsopn, amscd, mathtools, array, enumerate, amsfonts, graphics, hyperref, wasysym, graphicx}
\usepackage{float} 
\usepackage{mathrsfs}
\newtheorem{theorem}{Theorem}[section]
\newtheorem{corollary}[theorem]{Corollary}
\newtheorem{lemma}[theorem]{Lemma}

\newtheorem{lemma and definition}[theorem]{Lemma and Definition}
\newtheorem{proposition}[theorem]{Proposition}
\newtheorem{definition}[theorem]{Definition}
\newtheorem{exam}[theorem]{Example}

\newtheorem{remark}[theorem]{Remark}

\newtheorem{the construction}[theorem]{THE CONSTRUCTION}

\numberwithin{equation}{section}
\addcontentsline{toc}{section}{Acknowledgement}
\usepackage[T1]{fontenc}
\DeclareFontFamily{T1}{calligra}{}
\DeclareFontShape{T1}{calligra}{m}{n}{<->s*[1.44]callig15}{}
\DeclareMathAlphabet\mathrsfso      {U}{rsfso}{m}{n}
\usepackage{footnote}
\usepackage{multicol}
\usepackage{booktabs}
\usepackage[flushleft]{threeparttable}
\usepackage{tikz}
\usepackage{algpseudocode}
\usepackage[linesnumbered,ruled,vlined]{algorithm2e}
\usepackage{algpseudocode}
\SetKwInput{KwInput}{Input}                
\SetKwInput{KwOutput}{Output}              
\SetKwInput{KwInitialize}{Initialize} 
\SetKwInput{KwProcedure}{Procedure} 

\makeatletter
\newcommand*\rel@kern[1]{\kern#1\dimexpr\macc@kerna}
\newcommand*\widebar[1]{%
	\begingroup
	\def\mathaccent##1##2{%
		\rel@kern{0.8}%
		\overline{\rel@kern{-0.8}\macc@nucleus\rel@kern{0.2}}%
		\rel@kern{-0.2}%
	}%
	\macc@depth\@ne
	\let\math@bgroup\@empty \let\math@egroup\macc@set@skewchar
	\mathsurround\z@ \frozen@everymath{\mathgroup\macc@group\relax}%
	\macc@set@skewchar\relax
	\let\mathaccentV\macc@nested@a
	\macc@nested@a\relax111{#1}%
	\endgroup
}
\makeatother

\newcommand{\field}[1]{\mathbb{#1}}

\newcommand{\Z }{\field{Z}}
\newcommand{\N }{\field{N}}

\DeclareMathOperator{\ap}{Ap}
\DeclareMathOperator{\pf}{PF}
\DeclareMathOperator{\sg}{SG}
\DeclareMathOperator{\e}{e}
\DeclareMathOperator{\m}{m}

\DeclareMathOperator{\co}{C}
\DeclareMathOperator{\fr}{F}
\DeclareMathOperator{\G}{G}
\DeclareMathOperator{\arfg}{ArfG}
\DeclareMathOperator{\arf}{Arf}
\DeclareMathOperator{\msgArf}{Arfmsg}

\title{IRREDUCIBILITY OF ARF NUMERICAL
	SEMIGROUPS}
\author{Meral S\"{u}er }
\address{Department of Mathematics, Batman University, Batman, Turkey}
\email{meral.suer@batman.edu.tr}

\begin{document}
\maketitle

\begin{abstract}
In this paper, we introduce the concept of Arf special gaps of an Arf numerical semigroup, and an algorithm for computing all Arf special gaps of a given Arf numerical semigroup. We introduce the concept of Arf-irreducible numerical semigroups, and draw conclusions about all these concepts. We give a system of generators for the Frobenius variety and variety of families of Arf numerical semigroups. Moreover, we obtain the minimal Arf system of generators of a given Arf numerical semigroup.
\end{abstract}

\section{Introduction}

Let $\N$ and $\Z$ denote the set of non-negative integers and integers, respectively. A subset $H$ of $\N$ is a numerical semigroup if it is closed under addition, has a finite complement in $\N$ and  $0\in H$. Given  $X=\left\lbrace x_1,\dots,x_k\right\rbrace \subseteq \N$, $\left\langle X\right\rangle $ will denote the submonoid of $\N$ generated by $X$; that is,
$$\langle  X \rangle =\langle x_1,\dots,x_k \rangle = \{ n_1 x_1 +\dots+n_k x_k:n_i \in \N \}.$$

If  $X\subset H$ and $H=\langle  X \rangle $, $X$ is called a system of generators of $H$. If no proper subset of such $X$ generates  $H$, then $X$ is a minimal system of generators of $H$. Every numerical semigroup admits a unique  minimal system of generators. If $X=\left\lbrace x_1,\dots,x_k\right\rbrace$ is the minimal system of generators of $H$, then $k$ is called the embedding dimension of $H$, denoted by $\operatorname{e}(H)$, and $x_1$ is called the multiplicity of $H$ and denoted by $\m(H)$ which is the least positive integer in $H$. It is known that $\e(H) \leq \m(H)$  \cite[Chapter 1]{Ros09}. The numerical semigroup $H$ is said to have maximal embedding dimension if $\e(H)=\m(H)$.

If $H$ is a numerical semigroup, we assume $H=\left\lbrace h_0=0,h_1,\dots,h_n,\rightarrow\right\rbrace $ where “ $\rightarrow $” means that all subsequent natural numbers which are bigger then $h_n$, belong to $H$ and $\{h_0=0,h_1,\dots,h_{n} \}$ is the set of small elements of $H$.

A numerical semigroup $H$ is Arf if for all $h_i,h_j,h_k\in H$, with $h_i\geq h_j\geq h_k$, we have $h_i+ h_j-h_k\in H$. This property is equivalent to $2h_i- h_j\in H$ for all $h_i,h_j\in H$ with $h_i\geq h_j$. It is well known that every Arf numerical semigroup has maximal embedding dimension.

We say  $x\in\N $ is a gap of $H$ if $x\notin H$. The largest gap of $H$ is called the Frobenius number of $H$, denoted $\fr(H)$. If $H=\N$, then $\fr(H)$ is defined to be $-1$. The set of all gaps of $H$ is denoted by $\G(H)$. Moreover, if $x \in \G(H)$ and  $k$ is a non-negative integer such that $k$ divides  $x$, then $k\in \G(H)$. The smallest element of $H$, to which all subsequent natural numbers belong to $H$ is called the conductor of $H$, denoted by  $\co(H)$. Clearly, $\co(H) = \fr(H) + 1$ and $\co(\N)=0$. We have $\co(H)\geq 2$ if and only if $H \neq \N$.
The Ap\'ery set of $H$ with respect to a nonzero $h\in H$ is defined as
$$ \ap(H,h)=\{x \in H : x - h\not \in H\}.$$
With the help of previous studies, we can express the Ap\'ery set as follows. The Ap\'ery set $ \ap(H,h)$  constitutes a complete set of residues modulo $h$ and we shall write  $w(i)$ for the unique $w(i)\in\ap(H,h)$ satisfying  $w(i) \equiv i \ ($mod \ $h)$ \cite[Chapter 1]{Ros09}.
It is also quite simple to see the following
\begin{itemize}
	\item $H = \langle  h, w(1), \ldots , w(h-1) \rangle$,
	\item $\fr(H) = {\mbox{max(Ap}}(H,h)) - h$,
	\item $\co(H)= {\mbox{max(Ap}}(H,h))- h +1$,	
	\item $ {\mbox{max(Ap}}(H,h))=\co(H) + h -1.$	
	
\end{itemize}
Moreover, we define the following order relation in $\Z$ associated to $H :$ for $x,y\in\Z$, $x\leqslant_{H} y$ if $y-x\in H$. The set of maximal elements of $\Z\setminus H $ with respect to $\leqslant_{H}$ is the set of pseudo-Frobenius numbers of $H$, denoted by $\pf(H)$. The elements of the set $\sg(H) =\{x \in \pf(H) :2x\in H\} $ are called the special gaps of $H$.

Let $H$ be a numerical semigroup. A numerical semigroup $H$ is symmetric if and only if for all $x\in \Z\setminus H$, we have $\fr(H)-x\in H$. A numerical semigroup $H$ is pseudo-symmetric if and only if $\fr(H)$ is even and for all  $x\in \Z\setminus H$, either
$\fr(H)-x\in H$ or $x=\frac{\fr(H)}{2}$.

Numerical semigroup studies are always up-to-date and have applications in many fields. Some of the research on numerical semigroups are based on algebraic curves studied in algebraic geometry. Numerical semigroups are important because of their application to coding theory and algebraic curves (see for example \cite{Farran}). There is also a large literature on combinatorial aspects of numerical semigroups. Arf solved the classification problem of branches of a singular curve, depending on their multiplicity sequences in \cite{Arf}. The concept of Arf numerical semigroup is one of the important concepts emerging in the numerical semigroup theory. Although relevance of numerical semigroups with maximal embedding dimension  occur naturally among other numerical semigroups, it has gained a special reputation due to current applications to commutative algebra through its associated semigroup rings (see for eaxamples \cite{Barucci,Lipman}). The consideration of some properties of an analytically unramified one-dimensional local domain is accomplished using their value semigroups. Lipman introduces the Arf closure of the coordinate ring of the curve, and then its value semigroup (which is an Arf numerical semigroup) in \cite{Lipman}. With the help of Arf sequences, in \cite{H.Ibrahim}, the set of Arf numerical semigroups with given genus and conductor are calculated. Important conclusions about the minimal Arf system of generators of Arf numerical semigroups can also be reached in \cite{Branco} . 

In this research, we introduce the concept of Arf special gaps of a given Arf numerical semigroup which are dual to the concept of Arf minimal generators. With the aid of the Arf special gaps, we describe a way to calculate all Arf numerical semigroups that contain a given Arf numerical semigroup. We define Arf-irreducible semigroups which is the natural restriction of irreducibility to the variety of Arf numerical semigroups. The concept of Arf-irreducible semigroup is not the same as that the numerical semigroup is both Arf and irreducible. A numerical semigroup is called irreducible if it can not be expressed as an intersection of two numerical semigroups containing it properly which are introduced in \cite{Ros08}. We describe the varieties of families of Arf numerical semigroups and search for generators of these varieties. We also obtain the minimal Arf system of generators of a given Arf numerical semigroup.

Procedures presented in this manuscript  have been implemented by Pedro A. García Sánchez in \href{https://www.gap-system.org/}{ {\bf GAP}} \cite{GAP} and available sections 8.2.9 through 8.2.12 of  1.3.0 dev version of the package  \href{https://gap-packages.github.io/numericalsgps}{ {\bf numericalsgps} } \cite{Delgado}. 

\section{Arf special gaps of an Arf numerical semigroup} {\label{sec 3}}
In this section, we introduce the concept of Arf special gaps of an Arf numerical semigroup which are dual to the concept of Arf minimal generators in Arf characters. The motivation behind this definition is to find the set of all Arf numerical semigroups containing a given Arf numerical semigroup. We also present an algorithm for finding Arf special gaps of a given Arf numerical semigroup and get some results related to this concept.
\vspace*{0.2cm}
\begin{proposition} \cite[Proposition 9]{Ros10} {\label{Pro. 3}}  Let $H$ be a numerical semigroup and let $x\in \G(H) $. The following statements are equivalent:
	\begin{enumerate}
		\item 	$x\in \sg(H)$,
		\item  	 $H\cup \{x\} $ is a numerical semigroup.	
	\end{enumerate} 
\end{proposition}

\begin{proposition}\cite[Lemma $4.35$]{Ros09} {\label{Pro. 1}}  Let $H_1$ and $H_2$ be two numerical semigroups such that $H_1\subsetneq H_2$. Then $H_1\cup \left\lbrace {\mbox{max}}(H_2\setminus H_1)) \right\rbrace $ is a numerical semigroup, or equvalently, $\mbox{max}(H_2\setminus H_1)\in \sg(H_1)$.
\end{proposition}
\begin{definition} {\label{Def. 1}} Given an Arf numerical semigroup $H$, denote by
	$$	\arfg(H) =\{x\in \sg(H):H\cup\{x\} \textrm{ is an Arf numerical semigroup} \}.$$
	The elements of this set are called Arf special gaps of $H$.	
\end{definition}
\begin{proposition} {\label{Pro. a}}  Let $H=\{h_0=0,h_1,\dots,h_{n-1},h_n,\rightarrow \}$ be an Arf numerical semigroup and $x\in \G(H)$. Assume that $h_i<x<h_{i+1}$, $0\leqslant i\leqslant n-1$. Then $x\in \arfg(H) $ if and only if $x\in \sg(H)$, $2x-h_i\in H$ and $2h_{i+1}-x\in H$.
\end{proposition}
\begin{proof} Assume that $x\in \arfg( H)$. Then $H\cup\left\lbrace x\right\rbrace$ is an Arf numerical semigroup. Hence $x\in \sg(H)$ by Proposition {\ref{Pro. 3}}. Moreover, $2x-h_i, 2h_{i+1}-x\in H\cup \{x\}$ since $H\cup \{x\}$ is Arf. Since $2x-h_i$ and $2h_{i+1}-x$ are different from $x$, $2x-h_i$ and $2h_{i+1}-x\in H$ are elements of $H$. Conversely, assume that $x\in \sg(H)$; $2x-h_i\in H$ and $2h_{i+1}-x\in H$. So, $H\cup \{x\} $ is a numerical semigroup by Proposition {\ref{Pro. 3}}, and for  $j\geq k$,  $2h_j-h_k\in H$ since $H$ is Arf. If $j<i$, then	
	$$2x-h_j=(2x-h_i)+(2h_i-h_j)-h_i\in H;$$ 
	since $2x-h_i>x>h_i$, $2h_i-h_j>h_i$ and $H$ is Arf.
	
	If $j>i+1$, then
	$$2h_j-x=(2h_j-h_{i+1})+(2h_{i+1}-x)-h_{i+1}\in H;$$ 
	since $2h_j-h_{i+1}>h_{i+1}$, $2h_{i+1}-x>h_{i+1}$ and $H$ is Arf.	
	
	Thus $H\cup \{x\}$ is an Arf numerical semigroup, and $x\in \arfg( H)$.
	\end{proof} 	

Algorithm  {\ref{CHalgorithm}} provides the procedure to calculate all Arf special gaps of a given Arf numerical semigroup.

\begin{algorithm}[H]
	\label{CHalgorithm}
	\DontPrintSemicolon
	
	\KwInput{An Arf Numerical Semigroup $H=\{h_0=0,h_1,\dots,h_n,\rightarrow \}	$}
	\KwOutput{The Set of Arf special gaps of $H$, $\arfg(H)$}
	\KwInitialize{$\arfg(H)= \varnothing $}
	\KwProcedure{The following}
	
	Calculate $\sg( H)=\{e_1,\dots,e_r\}$
	\\	List $h_i<e_j<h_{i+1}$ for $i\in\{0,1,\dots n-1 \}$ and  $j\in\{1,\dots, r \}$
	\\ If $2e_j-h_i\in H$ and $2h_{i+1}-e_j\in H$, then add $ e_j$ to $\arfg(H)$\\
	Output $\arfg(H)$, and stop
	\caption{Algorithm for computing Arf special gaps of an Arf numerical semigroup}
\end{algorithm}
\begin{exam}{\label{exam. 2}} {\label{exam. 2}} Let $H$ be the Arf numerical semigroup
	
$$H=\langle10,24,25,26,27,28,29,31,32,33\rangle. $$ 
Let us compute $\arfg(H)$ using Algorithm  {\ref{CHalgorithm}}, we get
	$$H=\langle10,24,25,26,27,28,29,31,32,33\rangle=\{0,10,20,24,\rightarrow \}$$
	\begin{enumerate}
		\item $\sg( H)=\{14,15,16,17,18,19,21,22,23\}$
		\item  $10\leq14\leq20$, $10\leq15\leq20$, $10\leq16\leq20$, $10\leq17\leq20$, $10\leq18\leq20$, $10\leq19\leq20$, $20\leq21\leq24$, $20\leq22\leq24$ and $20\leq23\leq24$.
		\item
		\begin{enumerate}
			\item Since $2\cdot14-10=18\notin H$, $14$ can not be added into $ \arfg( H)$.
			\item Since $2\cdot15-10=20\in H$ and  $2\cdot20-15=25\in H$, we add $15$ into $\arfg( H)$.
			\item Since $2\cdot16-10=22\notin H$, $16$ can not be added into $ \arfg( H)$.
			\item Since $2\cdot20-17=23\notin H$, $17$ can not be added into $ \arfg( H)$.
			\item Since $2\cdot20-18=22\notin H$, $18$ can not be added into $ \arfg( H)$.
			\item Since $2\cdot20-19=21\notin H$, $19$ can not be added into $ \arfg( H)$.
			\item Since $2\cdot21-20=22\notin H$, $21$ can not be added into $ \arfg( H)$.
			\item Since $2\cdot22-20=24\in H$ and  $2\cdot24-22=26\in H$, we add $22$ into $\arfg( H)$.
			\item Since $2\cdot23-20=26\in H$ and  $2\cdot24-23=25\in H$, we add $23$ into $\arfg( H)$.
		\end{enumerate}
	\end{enumerate}	
	Thus, we find $\arfg( H)=\{15,22,23\}$. Furtheremore,
	$H\cup \{15\}=\{0,10,15,20,24,\rightarrow \}$, $H\cup \{22\}=\{0,10,20,22,24,\rightarrow \}$ and $H\cup \{23\}=\{0,10,20,23,\rightarrow \}$ are Arf numerical semigroups.	
	
	The manual operations given above can be done in the \href{https://www.gap-system.org/}{ {\bf GAP}}  session with the 
	\href{https://gap-packages.github.io/numericalsgps}{ {\bf numericalsgps} } package as follows:                                          
 \begin{verbatim}
	gap> s:=NumericalSemigroup(10,24,25,26,27,28,29,31,32,33);;
	gap> ArfSpecialGaps(s);
	[ 15, 22, 23 ]
\end{verbatim}
\end{exam} 
\begin{proposition} {\label{Pro. 7}} Let $H_1$ and $H_2$ be two Arf numerical semigroups such that $H_1\subsetneq H_2$. Then $H_1\bigcup \{ \max (H_2\setminus H_1)\} $ is an Arf numerical semigroup, or equivalently, $\max (H_2\setminus H_1) \in \arfg(H_1)$.
\end{proposition} 
\begin{proof} Let $H_1$ and $H_2$ be two Arf numerical semigroups such that $H_1\subsetneq H_2$, and let $\mbox{max}(H_2\setminus H_1)=x$. Thus  $H_1\cup \{x\} $ is a numerical semigroup by Proposition {\ref{Pro. 1}}. Consider $a,b\in H_1\cup \{x\} $ with $a>b$. If $a=x$, then $2a-b=2x-b\in H_2$ as $H_2$ is Arf. Since $2x-b>x$, we conclude that $2a-b\in H_1$ as $\mbox{max}(H_2\setminus H_1)=x$. Similarly, if $b=x$, then $2a-b=2a-x\in H_1$; $2a-x>x$ and thus $2a-b\in H_1$. Finally, if $a\neq x$ and $b\neq x$, then $a,b\in H_1$. So $2a-b\in H_1$ as $H_1$ is Arf. Hence $2a-b\in H_1\cup \{x\}$ for any $a,b\in H_1\cup \{x\}$ with $a\geq b$, and thus $H_1\cup \{x\}$ is an Arf numerical semigroup.
\end{proof} 
Given an Arf numerical semigroup $H$, we use $\mathrsfso{A}(H)$ to denote the set of all Arf numerical semigroups containing $H$. Since the complement of $H$ in $\N$ is finite, $\mathrsfso{A}(H)$ is finite. Given two Arf numerical semigroups $H$ and $\overline{H}$ with  $H\subseteq \overline{H}$, we define recursively
\begin{enumerate}
	\item $H_0 =H$;	
	\item \(
	H_{n+1}= \left\{ \begin{array}{ll}
		H_{n}\cup\{\max(\overline{H}\setminus H_n)\}  & \textrm{ if  $H_{n}\neq \overline{H};$}\\
	H_{n} & \textrm{ otherwise.}
	\end{array} \right.
	\)
	
	If the cardinality of  $ \overline{H}\setminus H $  is $r$, then
	$$H=H_0\subsetneq H_1\subsetneq\dots\subsetneq H_{r}=\overline{H}.$$
\end{enumerate}
With this idea, we can form the set $\mathrsfso{A}(H)$.  We begin with $\mathrsfso{A}(H)=\{ H\} $, and then for each element not equal to $\N$ in $\mathrsfso{A}(H)$ (note that $\arfg(\N)$  is the empty set), we attach to $\mathrsfso{A}(H)$  the  Arf numerical semigroups $H\cup\{ x \} $ with $x$ ranging in $\arfg( H)$.
\begin{exam}{\label{exam. 1}} 	(The Following are obtained by using the \href{https://www.gap-system.org/}{ {\bf GAP}}  package
\href{https://gap-packages.github.io/numericalsgps}{ {\bf numericalsgps} }).	Let $H=\langle 6,9,11,13,14,16\rangle=\{0,6,9,11,\rightarrow\}.$ 
	\cite{Delgado}
	\begin{verbatim}
	gap> H:= NumericalSemigroup(6,9,11,13,14,16);
	<Numerical semigroup with 6 generators>
	gap> IsArfNumericalSemigroup(H);
	true
	gap> se:=SmallElements(H);
	[ 0, 6, 9, 11]
	gap> SpecialGaps(NumericalSemigroup(6,9,11,13,14,16));
	[ 3, 7, 8, 10]	
	\end{verbatim}
	Using Algorithm  {\ref{CHalgorithm}} we get $\arfg( H)=\{ 3,10 \}$, and so $H\cup \{3\} =\langle3,11,13\rangle=\{0,3,6,9,11,\rightarrow\}$ and $H\cup \{10\} =\langle6,9,10,11,13,14\rangle=\{ 0,6,9,\rightarrow\}$ are Arf numerical semigroups containing $H$ (the only one that differs in just one element).  As $\arfg( H\cup \{3\})=\{ 10 \}$, we obtain a new Arf numerical semigroup $H\cup \{3,10\}$ containing $H\cup \{3\}$. Similarly, as $\arfg( H\cup \{10\})=\{ 3,8 \}$, we get two new Arf numerical semigroups $H\cup \{3,10\}$ and $H\cup \{8,10\}$ containig $H\cup \{10\}$.  If we repeat this process, we get $\mathrsfso{A}(H)=\{ H =\langle 6,9,11,13,14,16\rangle,  \langle  6,9,10,11,13,14\rangle,\\
	 \langle  6,8,9,10,11,13\rangle,
	  \langle  6,7,8,9,10,11\rangle, \langle  5,6,7,8,9\rangle,\langle  4,6,9,11\rangle,\langle  4,6,7,9\rangle, \\
	  \langle  4,5,6,7\rangle,\langle 3,11,13\rangle,
	   \langle 3,10,11\rangle,\langle  3,8,10\rangle,\langle  3,7,8\rangle,\langle  3,5,7\rangle, \langle  3,4,5\rangle,\langle  2,9\rangle,\\
	   \langle  2,7\rangle, \langle  2,5\rangle,\langle  2,3\rangle,\N =\langle 1\rangle\}$, which we draw below as a graph.
	
\begin{figure}[H]
	\centering
	\includegraphics[width=1.0\textwidth]{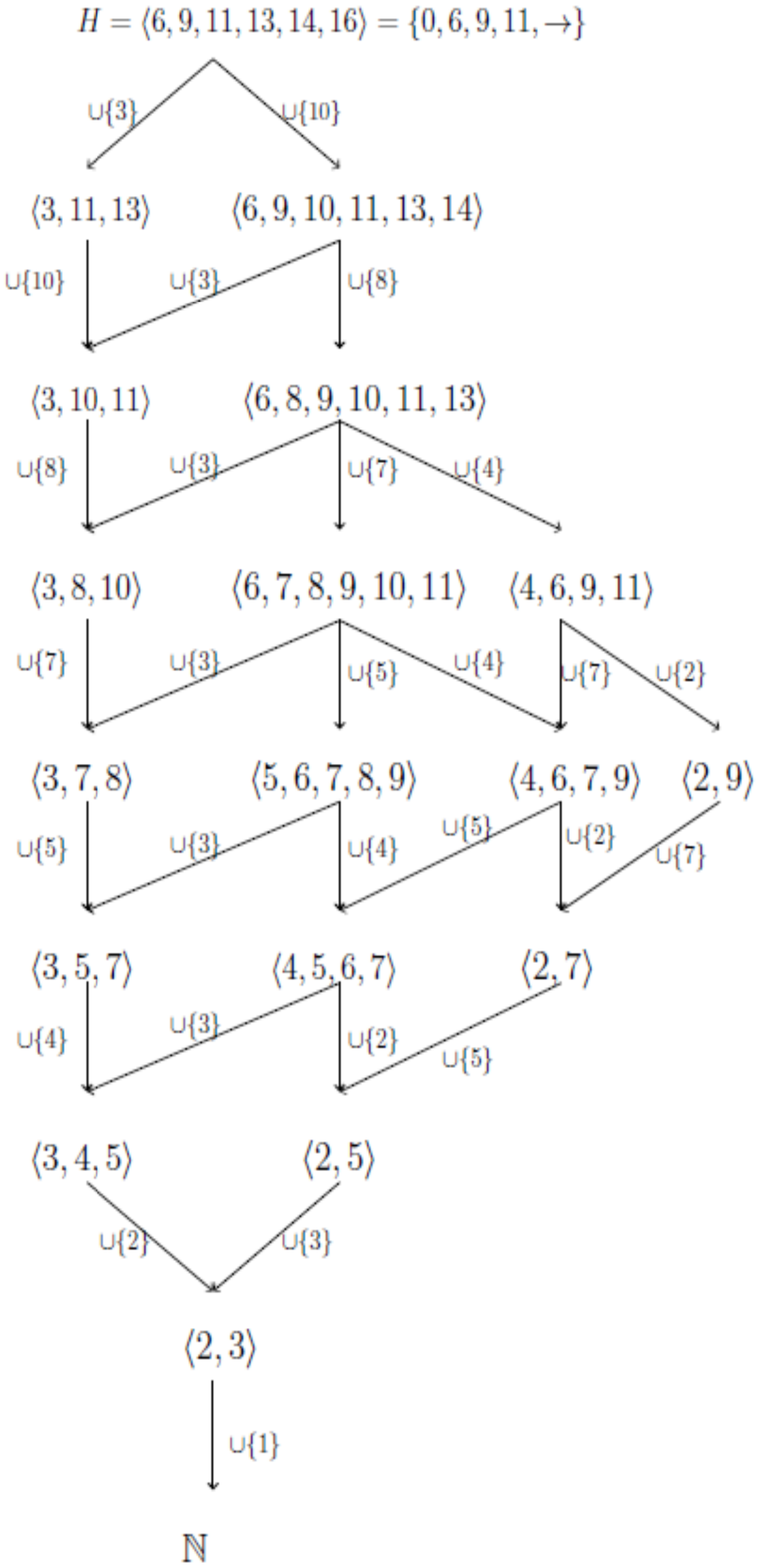}
	\caption{ All Arf numerical semigroups containing $H=\langle 6,9,11,13,14,16\rangle$.\label{fig0}}
\end{figure}
The list of all Arf numerical semigroups containing given Arf numerical semigroup H  in Figure {\ref{fig0}} is calculated as follows in the \href{https://www.gap-system.org/}{ {\bf GAP}}  session with \href{https://gap-packages.github.io/numericalsgps}{ {\bf numericalsgps} } package:
\begin{verbatim}
gap> s:=NumericalSemigroup(6,9,11,13,14,16);;
gap> List(ArfOverSemigroups(s),MinimalGenerators);
[ [ 1 ], [ 2, 3 ], [ 2, 5 ], [ 2, 7 ], [ 2, 9 ], [ 3 .. 5 ], 
[ 3, 5, 7 ], [ 3, 7, 8 ], [ 3, 8, 10 ], [ 3, 10, 11 ],
[ 3, 11, 13 ], [ 4 .. 7 ], [ 4, 6, 7, 9 ], [ 4, 6, 9, 11 ], 
[ 5 .. 9 ], [ 6 .. 11 ], [ 6, 8, 9, 10, 11, 13 ],
[ 6, 9, 10, 11, 13, 14 ], [ 6, 9, 11, 13, 14, 16 ] ]
\end{verbatim}					
\end{exam} 
We conclude from Proposition {\ref{Pro. 7}} is that if $H$ is an Arf numerical semigroup, then in the set of all Arf numerical semigroups that do not contain any elements of $\arfg( H)$, $H$  is maximal (in terms of set inclusion). In addition, $\arfg( H)$ is the smallest set of gaps that determines $H$ up to maximality.
\begin{proposition} {\label{Pro. 8}} Let $H$ be an Arf numerical semigroup and $\{ g_1,..., g_t \}\subseteq \sg( H)$. The following are equivalent:
	\begin{enumerate}
		\item $H$ is maximal (with respect to set inclusion) in the set of all Arf numerical semigroups $\overline{H}$ with
		$\overline{H}\cap\{ g_1,..., g_t \}$ is empty.
		\item  $\arfg( H) \subseteq\{ g_1,..., g_t \} $.
	\end{enumerate}
\end{proposition}

\begin{proof} $(1)\Rightarrow (2)$ Let $x\in \arfg( H)$. According to Definition {\ref{Def. 1}}, $H\bigcup \{ x \} $  is an Arf numerical semigroup containing $H$ properly. Thus, if $(1)$ holds, then $(H\bigcup \{ x \})\cap\{ g_1,..., g_t \}\neq\emptyset $, which yields $x\in\{ g_1,..., g_t \}$.
	$(2)\Rightarrow (1)$ Let $\overline{H}$ be an Arf numerical semigroup such that $H\subsetneq\overline{H}$ and $\overline{H}\cap\{ g_1,..., g_t \}=\emptyset$. Take $ \max (\overline{H}\setminus H)=x $. By Proposition {\ref{Pro. 7}}, $H\bigcup \{ \max (\overline{H}\setminus H)\}= H\bigcup \{ x\}$ is an Arf numerical semigroup. From Definition {\ref{Def. 1}}, $x\in \arfg(H)$. Because of the hypothesis, $x\in \{ g_1,..., g_t \}$. But this contradicts the assumption of $\overline{H}$.
\end{proof}

Given a rational number $a$, denote by  $\lceil a \rceil =\min(\Z\cap[a,\infty ))$,  $\lfloor a \rfloor =\max(\Z\cap( -\infty,a] )$.

\begin{proposition} {\label{Pro. 9}} For the Arf numerical semigroup $H= \{ 0,c,\rightarrow \}$, we have
	$$\#(\arfg( H))=\lfloor \frac{\co(H)}{2} \rfloor $$
	(where $\#(A)$ represents the number of elements of $A$).
\end{proposition}

\begin{proof} Note that $H= \{ 0,c,\rightarrow \}=\langle
	c,c+1,\dots,2c-1\rangle$ and $\pf(H)=\{1,\dots,c-1\}$,  and $\sg(H) =\{x \in \pf(H) :2x\in H\}=\{\lceil \frac{\co(H)}{2}\rceil,\lceil \frac{\co(H)}{2}\rceil+1,\dots ,c-1\}$. We write that  $H\bigcup\{ x \}= \{ 0,x,c,\rightarrow \}$ for all $x\in \sg( H)$.  Since $2c-x>c$, $2c-x\in H\bigcup\{ x \}$. Then  $H\bigcup\{ x \}$ is an Arf numerical semigroup for all $x\in \sg( H)$. Hence, $ \arfg( H)= \sg( H)$, and so
	\\
	
	$\#(\arfg( H))=c-1-\lceil \frac{\co(H)}{2} \rceil +1=c-\lceil \frac{\co(H)}{2} \rceil=\lfloor\frac{\co(H)}{2} \rfloor. $
\end{proof}
\section{ Arf-irreducible Numerical Semigroups}	
In this section, we introduce the concept of Arf-irreducible numerical semigroups. We obtain some results related to this concept.
\begin{definition} {\label{Def. 2}} An Arf numerical semigroup is called Arf-irreducible if it can not be expressed as the intersection of two Arf numerical semigroups containing it properly.	
\end{definition}

\begin{exam} {\label{example A}} The numerical semigroup $T =\{0,10,18,20,\rightarrow\}$ is an Arf numerical semigroup that is not an Arf-irreducible. Because $\{0,10,14,18,20,\rightarrow\}$ and  $\{0,10,16,18,20,\rightarrow\}$ are two Arf numerical semigroups properly containing $T$, and $T=\{0,10,14,18,20,\rightarrow\}\cap\{0,10,16,18,20,\rightarrow\}$.

\end{exam}
Recall that a numerical semigroup is irreducible if it can not be expressed as the intersection of two numerical semigroups properly containing it. Obviously, every irreducible Arf numerical semigroup is an Arf-irreducible. However, an Arf-irreducible numerical semigroup need not be an irreducible numerical semigroup (see Example {\ref{example B}} below).

Given an Arf numerical semigroup $H\neq \N$. It follows from Proposition {\ref{Pro. a}} that $H\bigcup \{ \fr(H)\} $ is an Arf numerical semigroup.

\begin{theorem} {\label{TEO. 10}} For an Arf numerical semigroup $H$, the following conditions are equivalent:
	\begin{enumerate}
		\item $H$ is an Arf-irreducible,
		\item $H$  is maximal in the set of all Arf numerical semigroups with Frobenius
		number  $\fr(H)$,
		\item $H$ is maximal in the set of all Arf numerical semigroups that do not contain
		$\fr(H)$.
	\end{enumerate}
\end{theorem}

\begin{proof}$(1)\Rightarrow (2)$ Let $K$ be  an Arf numerical semigroup such that $H\subseteq K$  and $\fr(H)=\fr(K)$. Then $H=(H\cup \{ \fr(H)\} )\cap K$. Since $H$ is an Arf-irreducible, we deduce that $H=K$.\\
	$(2)\Rightarrow (3)$ Let $K$ be  an Arf numerical semigroup such that $H\subseteq K$ and $\fr(H)\notin K$. Then $ L=K\cup\{ \fr(H)+1,\rightarrow \} $ is an Arf numerical semigroup such that $H\subseteq L$ and $\fr(H)=\fr(L)$. Therefore, $H=L=K\cup\{ \fr(H)+1,\rightarrow \} $, and so $H=K$.\\
	$(3)\Rightarrow (1)$ Let $H_1$ and $H_2$ be two Arf numerical semigroups that contain $H$ properly. Then by hypothesis, $\fr(H)\in H_1$  and $\fr(H)\in H_2$. Hence, $\fr(H)\in (H_1\cap H_2)\setminus H$; $H\neq H_1\cap H_2$. So $H$ is an Arf-irreducible.
\end{proof}
As a corollary, we find another characterization of Arf-irreducible numerical semigroup.
\begin{corollary} {\label{Cor. 11}} Let $H$ be an Arf numerical semigroup such that  $H\neq \N$. Then $H$is an Arf-irreducible if and only if 	$\#(\arfg( H))=1 $.
\end{corollary}

\begin{proof} By Theorem {\ref{TEO. 10}}, $H$ is an Arf-irreducible numerical semigroup if and only if it is maximal in the set of Arf numerical semigroups whose intersection with $\{\fr(H)\}$ is empty. From Proposition {\ref{Pro. 8}}, this is the same as $\arfg( H) \subseteq\{ \fr(H) \}$. Clearly, $\fr(H)\in \arfg( H)$ whenever $H\neq \N$.
\end{proof}
\begin{exam} {\label{example B}} Consider the Arf numerical semigroup $H=\left\lbrace 0,10,17,20\rightarrow\right\rbrace $. It is easy to see that $\arfg( H)=\{19\}$. So $H$ is an Arf-irreducible. However, $H$ is not irreducible as $H_1=\{0,10,17,19\rightarrow\}$ and $H_2=\{0,10,17,18,20\rightarrow\}$ are numerical semigroups properly containing $H$ such that $H=H_1\cap H_2$.
	
Whether $H$is Arf-irreducible or irreducible is expressed in the \href{https://www.gap-system.org/}{ {\bf GAP}}  session with \href{https://gap-packages.github.io/numericalsgps}{ {\bf numericalsgps} } package as follows:
\begin{verbatim}
	gap> s:=NumericalSemigroupBySmallElements([0,10,17,20]);;
	gap> IsArfIrreducible(s);
	true
	gap> IsIrreducible(s);
	false
	\end{verbatim}
\end{exam}
Every symmetric or pseudo-symmetric numerical semigroup is irreducible. Irreducible numerical semigroups are either symmetric (when their Frobenius number is odd) or pseudo-symmetric (when their Frobenius number is even)\cite{Assi}.

\begin{theorem} \cite[Theorem  $1.4.2.$]{Barucci} {\label{simetrik arf}}  Let $H$ be a (numerical) semigroup. The numerical semigroup $H$ is symmetric and Arf if and only if $H=\langle 2,m\rangle$ with $m\geq1$ and $m$ odd.
\end{theorem}
\begin{theorem} {\label{pseu-simetrik arf}}\cite[Theorem  $1.4.5.$]{Barucci} Let $H$ be a (numerical) semigroup. The numerical semigroup $H$ is pseudo-symmetric and Arf if and only if $H$ is either $\langle 3,4,5\rangle$ or $\langle 3,5,7\rangle$.
\end{theorem}

\begin{corollary}{\label{arf irreducible-irreducible}} Let $H$ be an Arf semigroup. The numerical semigroup $H$ is both irreducible and Arf-irreducible if and only if $H$ is is one of $\langle 3,4,5\rangle$ or $\langle 3,5,7\rangle$ or $\langle 2,m\rangle$ with $m\geq1$ and $m$ odd.
\end{corollary}
	\begin{proof} Follows from Corollary \ref{Cor. 11}, Theorem \ref{simetrik arf}, and Theorem \ref{pseu-simetrik arf}.
	\end{proof}

In Section \ref{sec 3}, we have obtained a procedure to construct the set $\mathrsfso{A}(H)$ of all Arf numerical semigroups containing a given Arf numerical semigroup $H$. While performing this procedure we can distinguish those over semigroups with at most one Arf special gap, which in view of Corrolary \ref{Cor. 11} are those Arf-irreducible over Arf numerical semigroup of $H$.
		\begin{proposition} {\label{Pro. 13}} Every Arf numerical semigroup can be expressed as an intersection of Arf-irreducible numerical semigroups.
		\end{proposition} 
		\begin{proof}
			Let $H$ be an Arf numerical semigroup. If  $H$  is not an Arf-irreducible, then there exists  Arf numerical semigroups $H_1$  and $H_2$ properly containing $H$ such that $H_1\cap H_2=H$. If $H_1$  and $H_2$ are not Arf-irreducible numerical semigroups, we can write each of them as an intersection of two other Arf numerical semigroups properly containing it. Since every Arf numerical semigroup appearing in this procedure belongs to $\mathrsfso{A}(H)$ and $\mathrsfso{A}(H)$ has finite number of elements, this process can be repeated only finitely many times.
		\end{proof}
		\begin{remark}For any Arf numerical semigroup $H$, let
			$$	\mathrsfso{QA}(H)=\left\lbrace \overline{H}\in\mathrsfso{A}(H): \overline{H} \textrm{ is an Arf-irreducible} \right\rbrace. $$
			By Proposition 	{\ref{Pro. 13}},
			
			$$H=\bigcap_{\overline{H}\in \mathrsfso{QA}(H) }\overline{H}.$$
			From this intersection, numerical semigroups that are not minimal with respect to set inclusion can be removed. The resulting intersection is still equal to $H$. Thus we have Proposition {\ref{Pro. 14}},
		\end{remark}
		
		\begin{proposition} {\label{Pro. 14}} Let $H$ be an Arf numerical semigroup and $\mathrsfso{QA}(H)$ be defined
			as given before. Assume that 	$ Minimals_{\subseteq}(\mathrsfso{QA}(H))=\{H_1,\dots,H_r \}$. Then 
			$$H=H_1\cap\dots\cap H_r. $$ 
		\end{proposition}
		
		\begin{exam}{\label{exam. 2}} The semigroup $H=\langle 6,9,11,13,14,16\rangle=\{0,6,9,11,\rightarrow\}$ appearing in Example {\ref{exam. 1}} is an Arf semigroup. Let us verify Proposition {\ref{Pro. 14}} for this semigroup. First we compute the set 
			
			$\mathrsfso{QA}(H)=\{H\cup\{3\}, H\cup\{3,10\},H\cup\{3,8,10\},H\cup\{3,7,8,10\},H\cup\{3,5,7,8,10\},H\cup\{3,4,5,7,8,10\},H\cup\{2,4,8,10\}, H\cup\{2,4,7,8,10\},H\cup\{2,4,5,7,8,10\},H\cup\{2,3,4,5,7,8,10\},H\cup\{1,2,3,4,5,7,8,10\}\}.$	
			
			Moreover,
			$$ Minimals_{\subseteq}(\mathrsfso{QA}(H))=\{H\cup\{3\}, H\cup\{2,4,8,10\}\}.$$
			
			Finally,
			$$H=\langle 6,9,11,13,14,16\rangle=(H\cup\{3\})\cap(H\cup\{2,4,8,10\})=\langle 3,11,13\rangle\cap\langle 2,9\rangle.$$ 
			
			The operations given above can be found in  the \href{https://www.gap-system.org/}{ {\bf GAP}}  session with \href{https://gap-packages.github.io/numericalsgps}{ {\bf numericalsgps} } package as follows:
			\begin{verbatim}
			gap> s:=NumericalSemigroup(6,9,11,13,14,16);;
			gap> li:=DecomposeIntoArfIrreducibles(s);
			[ <Numerical semigroup with 2 generators>,
			 <Numerical semigroup with 3 generators> ]
			gap> List(li,MinimalGenerators);
			[ [ 2, 9 ], [ 3, 11, 13 ] ]
			gap> ForAll(li, IsArfIrreducible);
			true
			gap> Intersection(li)=s;
			true
			\end{verbatim}
		\end{exam}
		\begin{proposition} {\label{Pro. 15}} Let $H$ be an Arf numerical semigroup. If $H=H_1\cap\dots\cap H_r$ with $H_1,\dots, H_r\in \mathrsfso{QA}(H)$, then there exists $ \overline{H_1},\dots, \overline{H_r}\in Minimals_{\subseteq}(\mathrsfso{QA}(H))$ such that $H= \overline{H_1}\cap \dots \cap \overline{H_r}$.
		\end{proposition}
		\begin{proof} 
			For every $i\in \{1,\dots,r\} $, if $H_i\notin Minimals_{\subseteq}(\mathrsfso{QA}(H))$, then take $\overline{H_i}\in Minimals_{\subseteq}(\mathrsfso{QA}(H))$ such that $\overline{H_i}\subset H_i$.
		\end{proof}
		
		\begin{proposition} {\label{Pro. 16}} Let $H$ be an Arf numerical semigroup. If $H_1\cap\dots\cap H_r\in\mathrsfso{A}(H) $. The following are equivalent.
			\begin{enumerate}
				\item $H=H_1\cap\dots\cap H_r $,
				\item For all $x\in \arfg( H)$, there exist $i\in \{1,\dots,r\} $ such that $x\notin H_i$.
			\end{enumerate}
		\end{proposition}
		
		\begin{proof}$(1)\Rightarrow (2)$ Assume  that $H=H_1\cap\dots\cap H_r $. If $x\in \arfg( H)$, then  $x\notin H$. Consequently, $H$, $x\notin H_i$ for some $i\in \{1,\dots,r\}$.
			
			$(2)\Rightarrow (1)$ If $H\subsetneq H_1\cap\dots\cap H_r $, then from Proposition {\ref{Pro. 7}} and Definition {\ref{Def. 1}}, $x=\max [(H_1\cap\dots\cap H_r)\setminus H]$ is in $\arfg( H)$. Hence, $x\in H_1\cap\dots\cap H_r$, that is, $x$ is in all the  $H_i$. This contradicts the hypothesis.
		\end{proof}
	Proposition \ref{Pro. 16} is a special form of the Proposition 25 in \cite{Ros10} for Arf special gaps 
	\begin{definition}
	A nonempty family $\Psi$ of numerical semigroups is a variety  if the following holds:
	\begin{enumerate}
		\item if $H\in\Psi$ and $H\subseteq H’$, then $H’\in\Psi$,
		\item if $H,H’\in\Psi$ , then $H\cap H’\in\Psi$.
	\end{enumerate}
\end{definition}
\begin{definition}
	
	A nonempty family $\Psi$ of numerical semigroups is a Frobenius variety  if the following holds:
	\begin{enumerate}
		\item if $H,H’\in\Psi$ , then $H\cap H’\in\Psi$.
		\item if $H\in\Psi$ and $H\neq \N$, then $H\cup \left\lbrace \fr(H)\right\rbrace\in \Psi$,
	\end{enumerate}
\end{definition}

We denote by $\mathfrak{A}$ and $\mathfrak{S}$ the set of all Arf numerical semigroups and the set of all numerical semigroups, respectively.
	
	\begin{proposition}{\label{prop arf}} \cite[Chapter 6]{Ros09} The set $\mathfrak{A}$ is a Frobenius variety.
	\end{proposition} 
	\begin{proposition} {\label{prop arf2}}\cite[Chapter 2]{Ros09} The intersection of finitely many Arf numerical semigroups is an Arf numerical semigroup.
	\end{proposition} 
	Given $H$ is an Arf numerical semigroup , define
	$$\mathrsfso{A}(H)=\left\lbrace S\in \mathfrak{A}:H\subseteq S \right\rbrace.$$
	\begin{theorem}{\label{thm arf}} Let $H$ be an Arf numerical semigroup. The set of all Arf numerical semigroups containing $H$ is a Frobenius variety.
	\end{theorem}
	\begin{proof}  Assume that $H$ is an Arf numerical semigroup and  $\mathrsfso{A}(H)$ is the set of all Arf numerical semigroups containing $H$.
		\begin{enumerate}
			\item If $S,S’\in\mathrsfso{A}(H)$, then  $S,S’ $ are Arf numerical semigroups such that $H\subseteq S$ and $H\subseteq S’$. Clearly, $H\subseteq S\cap S’$. By Proposition {\ref{prop arf2}},  $S\cap S’$ is an Arf numerical semigroup. So, $S\cap S’\in\mathrsfso{A}(H)$.
			\item If $S\in\mathrsfso{A}(H)$  and $S\neq \N$, then $S$ is an Arf numerical semigroup such that $H\subseteq S$. $S\cup \left\lbrace \fr(S)\right\rbrace$ is known that Arf. Since $H\subseteq S\subseteq S\cup \left\lbrace \fr(S)\right\rbrace $, 	 $S\cup \left\lbrace \fr(S)\right\rbrace\in\mathrsfso{A}(H)$.
		\end{enumerate}
		Thus,  $\mathrsfso{A}(H)$ is a Frobenius variety.
	\end{proof}
	Given a family of numerical semigroups $\mathscr{F}$, we denote by
	
	$$\mathfrak{B}(\mathscr{F})=\bigcap_{S\in\mathscr{F}} S$$
	
	(if $\mathscr{F}$ is empty, just set $\mathfrak{B}(\mathscr{F})=\N$).
	Observe that $\mathfrak{B}(\mathscr{F})$ is a submonoid
	of $\N$, which in general does not have to be a numerical semigroup. However, if $\mathscr{F}$ has finitely many elements, then $\mathfrak{B}(\mathscr{F})$ is a numerical semigroup (since the intersection is clearly a monoid, and its complement in $\N$ is still
	finite; note that the Frobenius number of the resulting semigroup is the maximum of the Frobenius numbers of the semigroups involved in the intersection).
	
	\begin{lemma} {\label{lem arf1}}\cite[Lemma $6$]{Ros10} The intersection of varieties is a variety.
	\end{lemma}
	
	This result allows us to introduce the concept of variety generated by a
	family of numerical semigroups. Let $\mathscr{F}\subseteq\mathfrak{S}$. We denote by $\left\langle \mathscr{F}\right\rangle $ the intersection of all varieties containing $\mathscr{F}$. By Lemma {\ref{lem arf1}}, $\left\langle \mathscr{F}\right\rangle $ is a variety, and it is the smallest (with respect to set inclusion) variety containing $\mathscr{F}$. We say that $\mathscr{F}$ is a system of generators of $\left\langle \mathscr{F}\right\rangle $. A variety $\Psi$ is finitely generated if it has a finite system of generators.
	\begin{lemma} {\label{lem arf1}}\cite[Lemma $7$]{Ros10}
	Let $\mathscr{F}$ be a family of numerical semigroups. Then
		
		$$\left\langle \mathscr{F}\right\rangle =\left\lbrace S \in\mathfrak{S}: \mathfrak{B}(\mathscr{F})\subseteq S \right\rbrace . $$ 
		We say that a variety $\Psi$ is cyclic if there exists a numerical semigroup $S$ such that $\Psi=\left\langle \left\lbrace S \right\rbrace\right\rangle  $ .
	\end{lemma}
	\begin{theorem}{\label{thm arf 2}} Let $H$ be an Arf numerical semigroup and let $\mathrsfso{A}(H)$ the set of all Arf numerical semigroups containing $H$. Then $\left\langle \mathrsfso{A}(H) \right\rangle=\left\langle \left\lbrace H \right\rbrace\right\rangle $.
	\end{theorem}
	\begin{proof}  Assume that $H$ is an Arf numerical semigroup and  $\mathrsfso{A}(H)$ is the set of all Arf numerical semigroups containing $H$. Since $H$ has finite
		complement in $\N$, $\mathrsfso{A}(H)$ has finitely many elements. Using Lemma {\ref{lem arf1}} and the definition $\mathfrak{B}(\mathscr{\mathrsfso{A}(H)})$, we have
		\begin{eqnarray*}
			\left\langle \mathrsfso{A}(H) \right\rangle & = & \left\lbrace S \in\mathfrak{S}: \mathfrak{B}(\mathrsfso{A}(H) )\subseteq S \right\rbrace 
			\\
			& = &\{  S \in\mathfrak{S}: \bigcap_{H’\in\mathrsfso{A}(H)} H’\subseteq S \} 
			\\
			& = &  \left\lbrace S \in\mathfrak{S}:  H\subseteq S \right\rbrace 
			\\
			& = & \left\langle \left\lbrace H \right\rbrace\right\rangle . 
		\end{eqnarray*}
		
	\end{proof}
	\begin{theorem}{\label{thm arf 3}} Let $\varOmega$ be a finite variety whose elements Arf are numerical semigroups and let $H'=\bigcap_{H\in\varOmega} H$. Then
		$$\left\langle \varOmega \right\rangle=\left\langle \left\lbrace H' \right\rbrace\right\rangle .$$
	\end{theorem}
	\begin{proof}  Let $\varOmega=\{H:H \in \mathfrak{A}\}$ and let  $H'=\bigcap_{H\in\varOmega} H$.  Using Lemma {\ref{lem arf1}} and the definition $\mathfrak{B}(\mathscr{\varOmega})$, we have
		\begin{eqnarray*}
			\left\langle \varOmega \right\rangle & = & \left\lbrace S \in\mathfrak{S}: \mathfrak{B}(\varOmega )\subseteq S \right\rbrace 
			\\
			& = &\{  S \in\mathfrak{S}: \bigcap_{H\in\varOmega} H\subseteq S \} 
			\\
			& = &  \left\lbrace S \in\mathfrak{S}:  H' \subseteq S \right\rbrace 
			\\
			& = & \left\langle \left\lbrace H'  \right\rbrace\right\rangle  .
		\end{eqnarray*}
		
	\end{proof}
\section{Minimal Arf systems of generators of Arf numerical semigroups}
If $X\subset\N$  and $gcd(X)=1$, then  $\left\langle X\right\rangle $ is a numerical semigroup. Any Arf numerical semigroup containing $X$ contains $\left\langle X\right\rangle$. The Arf closure of $\left\langle X\right\rangle$ is the smallest Arf numerical semigroup containing $\left\langle X\right\rangle$ (in terms of set inclusion), by denoted $\arf(\left\langle X\right\rangle )$. When $H$ is an Arf numerical semigroup, we can observe that $\arf(H)=H$. If $\arf(\langle X\rangle)=H$, we say that $X$ is an Arf system of generators of $H$, and we will say that $X$ is minimal if no proper subset of $X$ is an Arf system of generators of $H$. In this section, we will obtain the minimal Arf system of generators of a given Arf numerical semigroup. It is known that every Arf numerical semigroup has a unique minimal Arf system of generators \cite{Branco}. We will denote by $\msgArf(H)$ the minimal Arf system of generators of a given Arf numerical semigroup $H$.

\begin{lemma} {\label{branco 8}}\cite[Lemma $8$]{Branco} Let $H$ be an Arf numerical semigroup and let $a\in H$. The following conditions
	are equivalent:
	\begin{enumerate}
		\item $a$ belongs to the minimal Arf system of generators of $H$,
		\item $H\backslash \{a\}$ is an Arf numerical semigroup.
	\end{enumerate}
\end{lemma}

\begin{corollary}{\label{Ros3.19}} \cite[Chapter $2$]{Ros09} Let $H$ be a proper subset of $\N$. Then $H$ is an Arf numerical semigroup
	if and only if there exist positive integers $x_{1},...,x_{n}$ such that
	$$H=\{0,x_{1},x_{1}+x_{2},\dots,x_{1}+\dots+x_{n-1},x_{1}+\dots+x_{n},\rightarrow\}$$
	and  
	$x_{j}\in \{x_{j+1},x_{j+1}+x_{j+2},\dots,x_{j+1}+\dots+x_{n},\rightarrow\}$ for all $j\in \{1,\dots, n\}$.
\end{corollary}
Throughout this section, we assume that $H=\{h_0=0,h_1,\dots,h_n,\rightarrow \}=\{0,x_{1},x_{1}+x_{2},\dots,x_{1}+\dots+x_{n-1},x_{1}+\dots+x_{n},\rightarrow\}$ with $h_{j}=x_{1}+\dots+x_{j}$ for all $j\in \{1,\dots, n\}$ when $H$ is an Arf numerical semigroup, under the conditions of Corollary {\ref{Ros3.19}}
\begin{theorem}{\label{üretec 1}} Let $H$ be an Arf numerical semigroup expressed as in Corollary {\ref{Ros3.19}} and let $h_{i}\in H$ with $h_{i}\leq \fr(H)+1$. Then
	$h_{i}$ is in minimal Arf system of generators of $H$ if and only if for all $j<i$, $x_{j}\neq x_{j+1}+\dots+x_{i}$.
\end{theorem}
\begin{proof} Let $H$ be an Arf numerical semigroup expressed as in Corollary {\ref{Ros3.19}}. Then $H=\{h_0=0,h_1,\dots,h_n,\rightarrow \}=\{0,x_{1},x_{1}+x_{2},\dots,x_{1}+\dots+x_{n-1},x_{1}+\dots+x_{n},\rightarrow\}$ with $h_{j}=x_{1}+\dots+x_{j}$ for all $j\in \{1,\dots, n\}$.\\
$(\Rightarrow)$ Assume that $h_{i}\in \msgArf(H)$ and  $x_{j}= x_{j+1}+\dots+x_{i}$  for any $j<i$. Take $h_{j-1},h_{j}\in H \backslash \{ h_{i}\}$ for $j<i$. Then $2h_{j}-h_{j-1}=2(x_{1}+\dots+x_{j})-(x_{1}+\dots+x_{j-1})=(x_{1}+\dots+x_{j})+x_{j}=(x_{1}+\dots+x_{j})+(x_{j+1}+\dots+x_{i})=h_{i}$. Thus, $2h_{j}-h_{j-1}\notin H \backslash \{ h_{i}\}$. We deduce that $H \backslash \{ h_{i}\}$ is not an Arf numerical semigroup, which contradicts the fact that $h_{i}$ is in minimal Arf system of generators of $H$.\\
$(\Leftarrow)$ Assume that $x_{j}\neq x_{j+1}+\dots+x_{i}$ for all $j<i$. At the same time, since $H$ is an Arf numerical semigroup, by Corollary {\ref{Ros3.19}} we have	$x_{j}\in \{x_{j+1},x_{j+1}+x_{j+2},\dots,x_{j+1}+\dots+x_{n},\rightarrow\}$ for all $j\in \{1,\dots, n\}$. This means that $x_{j}\in \{x_{j+1},x_{j+1}+x_{j+2},\dots,x_{j+1}+\dots+x_{i-1},x_{j+1}+\dots+x_{i+1},\dots,x_{j+1}+\dots+x_{n},\rightarrow\}$ for all $j<i$. Let us set $x’_{j}$ as follows
		\begin{displaymath}
		x’_{j}= \left\{ \begin{array}{ll}
			x_{j} & \textrm{for $j<i$ }\\
			x_{i}+x_{i+1} & \textrm{for $j=i$ }\\
			x_{j+1} & \textrm{for $j>i$}.
		\end{array} \right.
	\end{displaymath}
	In this a case, we have $$x’_{j}\in \{x’_{j+1},x’_{j+1}+x’_{j+2},\dots,x’_{j+1}+\dots+x’_{i},\dots,x’_{j+1}+\dots+x’_{n-1},\rightarrow\}$$ for all $j\in \{1,\dots, n-1\}$. By Corollary {\ref{Ros3.19}} again we can write an Arf numerical semigroup $H'$ with positive integers $x’_{1},x’_{2},\dots, x’_{n-1}$ as follows;
	\[
	\begin{array}{ll}
	H'&=\{h'_0=0,h'_1,\dots,h'_{n-1},\rightarrow \}\\
	&=\{0,x'_{1},x'_{1}+x'_{2},\dots,x'_{1}+\dots+x'_{i},\dots,x'_{1}+\dots+x'_{n-1},\rightarrow\}
    \end{array}
	\]
	with $h'_{j}=x'_{1}+\dots+x'_{j}$ for all $j\in \{1,\dots, n-1\}$. It is not hard to see that
		\begin{displaymath}
		h’_{j}= \left\{ \begin{array}{ll}
			h_{j} & \textrm{for $j<i$ }\\
			h_{j+1} & \textrm{for $j\geq i$}.
		\end{array} \right.
	\end{displaymath}
Therefore,
\[
\begin{array}{ll}
H'&=\{h'_0=0,h'_1,\dots,h'_{n-1},\rightarrow \}\\
&=\{h_0=0,h_1,\dots,h_{i-1},h_{i+1},\dots,h_{n}\rightarrow\}= H \backslash \{ h_{i}\}.
\end{array}
\] 
Hence, $H \backslash \{ h_{i}\}$ is an Arf numerical semigroup, and so $h_{i}\in\msgArf(H)$ as a result of Lemma {\ref{branco 8}}. 
\end{proof} 
\begin{theorem}{\label{üretec 2}} Let $H$ be an Arf numerical semigroup expressed as in Corollary {\ref{Ros3.19}}. Then
	$\fr(H)+2$ is in the minimal Arf system of generators of $H$ if and only if   $x_{j}\neq x_{j+1}+\dots+x_{n}+1$ for all $j\in \{1,\dots, n\}$.
\end{theorem}
\begin{proof} Let $H$ be an Arf numerical semigroup expressed as in Corollary {\ref{Ros3.19}}. Then $H=\{h_0=0,h_1,\dots,h_n,\rightarrow \}=\{0,x_{1},x_{1}+x_{2},\dots,x_{1}+\dots+x_{n-1},x_{1}+\dots+x_{n},\rightarrow\}$ with $h_{j}=x_{1}+\dots+x_{j}$ for all $j\in \{1,\dots, n\}$. Also note that $\fr(H)+2=x_{1}+\dots+x_{n}+1=h_{n}+1$. \\
$(\Rightarrow)$  Assume that $\fr(H)+\in \msgArf(H)$ and $x_{j}= x_{j+1}+\dots+x_{n}+1$  for any $j\leq n$. Now take  $h_{j-1},h_{j}\in H \backslash \{ \fr(H)+2\}$ for $j\leq n$. Then $2h_{j}-h_{j-1}=2(x_{1}+\dots+x_{j})-(x_{1}+\dots+x_{j-1})=(x_{1}+\dots+x_{j})+x_{j}=(x_{1}+\dots+x_{j})+(x_{j+1}+\dots+x_{n}+1)=h_{n}+1=\fr(H)+2$. Hence, $2h_{j}-h_{j-1}\notin H \backslash \{ \fr(H)+2\}$. We deduce that $H \backslash \{ \fr(H)+2\}$ is not an Arf numerical semigroup, which contradicts the fact that $\fr(H)+2$ is in minimal Arf system of generators of $H$.\\
$(\Leftarrow)$ Assume that $x_{j}\neq x_{j+1}+\dots+x_{n}+1$ for all  $j\in \{1,\dots, n\}$. Since $H$ is an Arf numerical semigroup, for all $j\in \{1,\dots, n\}$
$$x_{j}\in \{x_{j+1},x_{j+1}+x_{j+2},\dots,x_{j+1}+\dots+x_{n},x_{j+1}+\dots+x_{n}+2\rightarrow\}.$$
Let us set $x’_{n+1}=2$ and $x’_{j}=x_{j}$ for $i\leq n$. According to Corollary {\ref{Ros3.19}} and with  positive integers $x’_{1},...,x’_{n},x’_{n+1}$, we can write an Arf numerical semigroup $H'$ as follows,
\[
\begin{array}{ll}
	H'&=\{h'_0=0,h'_1,\dots,h'_{n},h'_{n+1},\rightarrow \}\\
	&=\{0,x'_{1},x'_{1}+x'_{2},\dots,x'_{1}+\dots+x'_{n},\dots,x'_{1}+\dots+x'_{n+1},\rightarrow\},
\end{array}
\]
with $h'_{j}=x'_{1}+\dots+x'_{j}$ for all $j\in \{1,\dots, n+1\}$. We can get that
\begin{displaymath}
	h’_{j}= \left\{ \begin{array}{ll}
		h_{j} & \textrm{for $j\leq n$ }\\
		h_{n}+2 & \textrm{for $j=n+1$}.
	\end{array} \right.
\end{displaymath}
Therefore,
\[
\begin{array}{ll}
	H'&=\{h'_0=0,h'_1,\dots,h'_{n+1},\rightarrow \}\\
	&=\{h_0=0,h_1,\dots,h_{i-1},h_{i+1},\dots,h_{n},h_{n}+2\rightarrow\}\\
&= H \backslash \{ h_{n}+1\}=H \backslash \{ \fr(H)+2\}.
\end{array}
\] 
Hence, $H \backslash \{ \fr(H)+2\}$ is an Arf numerical semigroup, and so $\fr(H)+2\in\msgArf(H)$ as a result of Lemma {\ref{branco 8}}. 
\end{proof} 
\begin{remark}  Let $H=\{0,h_{1},h_{2},\dots h_{n},\rightarrow \} $ be a numerical semigroup. Note that $\msgArf(H)\subseteq\{ h_{1},h_{2},\dots h_{n}\}\cup\{ h_{n}+1\}$.
\end{remark}
A tree whose elements have at most two children is called a binary tree. The root of a tree is the top most node of the tree that has no parent node. There is only one root node in every tree. A node that has no child is known as the leaf node. There can be multiple leaf nodes in a tree. In \cite[Section  $2$]{Branco}, a binary tree of Arf numerical semigroups is obtained using Corollary  {\ref{branco 10}. This tree is constructed by starting from the set of non-negative integers and then removing the minimal Arf generators greater than the Frobenius number of each Arf numerical semigroup. Next,  we characterize the Arf numerical semigroups found in the leaves of this tree.
	\begin{corollary}{\label{branco 10}}
		\cite[Corollary $10$]{Branco} Let $H$  be an Arf numerical semigroup. Then $H$  is a leaf if and only if the
		minimal Arf system of generators of $H$  does not contain elements greater than $\fr(H)$. 
	\end{corollary}
	\begin{corollary}{\label{leaf}} Let $H$ be an Arf numerical semigroup expressed as in Corollary {\ref{Ros3.19}}. Then $H$  is a leaf if and only if there exist $i,j\in \{1,\dots, n\}$ such that  $x_{i}= x_{i+1}+\dots+x_{n}$  and $x_{j}= x_{j+1}+\dots+x_{n}+1$.
	\end{corollary}
	\begin{proof} Assume that $H$ is an Arf numerical semigroup expressed as in Corollary {\ref{Ros3.19}}.
		
		$(\Rightarrow)$ If $H$  is a leaf, then the minimal Arf system of generators of $H$  does not contain elements greater than $\fr(H)$ by Corollary {\ref{branco 10}}. If $\fr(H)+1=x_{1}+\dots+x_{n}\notin \msgArf(H)$, then there exist $i\in \{1,\dots, n\}$ such that $x_{i}= x_{i+1}+\dots+x_{n}$ by Theorem {\ref{üretec 1}}.  If $\fr(H)+2=x_{1}+\dots+x_{n}+1\notin \msgArf(H)$, then there exist $j\in \{1,\dots, n\}$ such that $x_{j}= x_{j+1}+\dots+x_{n}$ by Theorem {\ref{üretec 2}}.  
		
		$(\Leftarrow)$ If there exist $i,j\in \{1,\dots, n\}$ such that  $x_{i}= x_{i+1}+\dots+x_{n}$  and $x_{j}= x_{j+1}+\dots+x_{n}+1$, then $x_{1}+\dots+x_{n}, x_{1}+\dots+x_{n}+1\notin \msgArf(H)$ by Theorem {\ref{üretec 1}} and Theorem {\ref{üretec 2}}. From  Corollary {\ref{leaf}}, $H$  is a leaf.
	\end{proof} 
Algorithm  {\ref{CHalgorithm 2}} provides the procedure to calculate the minimal Arf system of generators of a given Arf numerical semigroup.

\begin{algorithm}[H]
	\label{CHalgorithm 2}
	\DontPrintSemicolon
	
	\KwInput{An Arf Numerical Semigroup $H=\{h_0=0,h_1,\dots,h_n,\rightarrow \}	$}
	\KwOutput{The minimal Arf system of generators of of $H$, $\msgArf(H)$}
	\KwInitialize{$\msgArf(H)= \varnothing $}
	\KwProcedure{The following}
	
	Calculate $x_{0}=0$ and $x_{i}=h_{i}-h_{i-1}$ for $i\in \{1,\dots, n \}$
	\\ If $x_{j}\neq x_{j+1}+\dots+x_{i}$ for all $i\in \{1,\dots, n \}$ and $j\in \{0,\dots,i-1\}$, then add $x_{1}+\dots+x_{i}=h_{i}$ to $\msgArf(H)$.
	\\ If $x_{j}\neq x_{j+1}+\dots+x_{n}+1$ for all $j\in \{0,\dots,n\}$, then add $x_{1}+\dots+x_{n}+1=h_{n}+1$ to $\msgArf(H)$.\\
	Output $\msgArf(H)$, and stop
	\caption{Algorithm for computing the minimal Arf system of generators  of given an Arf numerical semigroup}
\end{algorithm}
\begin{exam}{\label{msg 2}}  Let $H$ be the Arf numerical semigroup
	$$H=\langle 8,20,27,29,30,31,33,34\rangle. $$ 
	Let us compute $\msgArf(H)$ using Algorithm  {\ref{CHalgorithm 2}}.

$$H=\langle 8,20,27,29,30,31,33,34\rangle=\{0,8,16,20,24,27,\rightarrow \}$$
Performing the steps of the above algorithm
	\begin{enumerate}
		\item Let us calculate step 1: $x_{0}=0,x_{1}=8,x_{2}=8,x_{3}=4,x_{4}=4,x_{5}=3$.
		\item  Let us calculate step 2:
		\begin{enumerate}
				\item If $i=1$, then $j=0$. Since $x_{0}\neq x_{1}$,\\ we get  $h_{1}=x_{1}=8\in \msgArf(H)$.
				\item If $i=2$, then $j\in \{0,1\}$. Since \\
				$x_{0}\neq x_{1}+x_{2}$,\\ $x_{1}=x_{2}$,\\ we get $h_{2}=x_{1}+x_{2}=16\notin \msgArf(H)$.
				\item If $i=3$, then $j\in \{0,1,2\}$. Since\\ $x_{0}\neq x_{1}+x_{2}+x_{3}$,\\ $x_{1}\neq x_{2}+x_{3},$ \\ $x_{2}\neq x_{3}$,\\ we get $h_{3}=x_{1}+x_{2}+x_{3}=20\in \msgArf(H)$.
				\item If $i=4$, then $j\in \{0,1,2,3\}$. Since\\ $x_{0}\neq x_{1}+x_{2}+x_{3}+x_{4},$ \\$x_{1}\neq x_{2}+x_{3}+x_{4}$,\\ $x_{2}= x_{3}+x_{4}$,\\ $x_{3}= x_{4}$,\\ we get $h_{4}=x_{1}+x_{2}+x_{3}+x_{4}=24\notin \msgArf(H)$.
				\item If $i=5$, then $j\in \{0,1,2,3,4\}$. Since\\ $x_{0}\neq x_{1}+x_{2}+x_{3}+x_{4}+x_{5}$,\\ $x_{1}\neq x_{2}+x_{3}+x_{4}+x_{5}$,  \\ $x_{2}\neq x_{3}+x_{4}+x_{5}$,\\ $x_{3}\neq x_{4}+x_{5}$,\\ $x_{4}\neq x_{5}$, \\we get $h_{5}=x_{1}+x_{2}+x_{3}+x_{4}+x_{5}=27\in \msgArf(H)$.	 
			\end{enumerate}
	\item  Let us calculate step 3: for $j\in \{0,1,\dots n\}$\\
	$x_{0}\neq x_{1}+x_{2}+x_{3}+x_{4}+x_{5}+1,$\\
	 $x_{1}\neq x_{2}+x_{3}+x_{4}+x_{5}+1,$\\
	 $x_{2}\neq x_{3}+x_{4}+x_{5}+1,$\\
	 $x_{3}\neq x_{4}+x_{5}+1,$\\
	 $x_{4}=x_{5}+1,$\\
	 we get $h_{5}+1=x_{1}+x_{2}+x_{3}+x_{4}+x_{5}+1=28\notin \msgArf(H)$.
	 \item  Thus $\msgArf(H)=\{8,20,27\}$
	 \end{enumerate}
	 	Let us make the computations with the \href{https://www.gap-system.org/}{ {\bf GAP}}  session with \href{https://gap-packages.github.io/numericalsgps}{ {\bf numericalsgps} } package.
	 	\begin{verbatim}
	 	gap> s:=NumericalSemigroup(8,20,27,29,30,31,33,34);;
	 	gap> ArfCharactersOfArfNumericalSemigroup(s);
	 	[ 8, 20, 27 ]
	 	\end{verbatim}

\end{exam} 
\begin{remark}  When we remove each minimal generator of the numerical semigroup given in Examle {\ref{msg 2}} from the numerical semigroup, the obtained numerical semigroups are  $H_{1}=\{0,16,20,24,27,\rightarrow \}$ and  $H_{2}=\{0,8,16,24,27,\rightarrow \}$ and $H_{3}=\{0,8,16,20,24,28,\rightarrow \}$. Moreever,  $\arfg( H_{1})=\{8,12,26\}$ and $\arfg( H_{2})=\{20,26\}$ and $\arfg( H_{3})=\{12,26,27\}$. Note that  $H_{1}\cup \{8\}=H $ and $H_{2}\cup \{20\}=H $ and $H_{3}\cup \{27\}=H $.
\end{remark}

\section*{Acknowledgement}
The author would like to thank Pedro A. García Sánchez for his carefully reading of an earlier version of the paper, for his helpful comments and suggestions, and for his implementing the presented procedures in the \href{https://www.gap-system.org/}{ {\bf GAP}}.

\end{document}